\documentclass{amsart}[12pt]
\usepackage{fullpage} 
\usepackage{mathrsfs,fontenc,enumitem}
\usepackage[alphabetic]{amsrefs}
\usepackage{graphicx}
\usepackage[mathcal]{euscript}
\usepackage{caption,subcaption}
\usepackage{tikz,tikz-cd}
\usetikzlibrary{arrows}

\newtheorem{thm}{Theorem}
\newtheorem{lem}[thm]{Lemma}
\newtheorem{prop}[thm]{Proposition}

\newtheorem{cor}[thm]{Corollary}
\newtheorem{defe}[thm]{Definition}
 
\theoremstyle{remark}
\newtheorem{rem}[thm]{Remark}
\newtheorem{exam}[thm]{Example}

\DefineSimpleKey{bib}{myurl}
\newcommand\myurl[1]{\url{#1}}
\BibSpec{webpage}{
  +{}{\PrintAuthors} {author}
  +{,}{ \textit} {title}
  +{}{ \parenthesize} {date}
  +{,}{ \myurl} {myurl}
}

\newcommand{\g}{ \mathfrak{g} }
\newcommand{\gl}{ \mathfrak{gl} }
\newcommand{\ad}{ \operatorname{ad} }
\newcommand{\Ad}{\operatorname{Ad}}
\newcommand{\Gal}{\operatorname{Gal}}
\newcommand{\GL}{\operatorname{GL}}

\newcommand{\Ct}{\mathbb{C}(\!(t)\!)}

\newcommand{\CT}{\mathbb{C}[\![t]\!]}
\newcommand{\CS}{\mathbb{C}[\![s]\!]}

\newcommand{\bQ}{\mathbb{Q}}
\newcommand{\cO}{\mathcal{O}}
\newcommand{\cU}{\mathcal{U}}
\newcommand{\bZ}{\mathbb{Z}}
\newcommand{\ra}{\rightarrow}
\newcommand{\cK}{\mathcal{K}}
\newcommand{\cKx}{\mathcal{K}\{x\}}

\newcommand{\bC}{\mathbb{C}}
\newcommand{\Ext}{\mathrm{Ext}}
\newcommand{\HH}{\mathrm{H}}

\newcommand{\fg}{\mathfrak{g}}
\newcommand{\Span}{\operatorname{span}}
\newcommand{\Hom}{\operatorname{Hom}}
\newcommand{\Nilp}{\mathrm{Nilp}}

\newcommand{\fh}{\mathfrak{h}}

\begin{document}
\title{Jordan decomposition for formal $G$-connections} 
\author{Masoud Kamgarpour and Samuel Weatherhog}
\email{masoud@uq.edu.au}
\email{samuel.weatherhog@gmail.com} 
\address{School of Mathematics and Physics, The University of Queensland, St. Lucia, Queensland, Australia}


\begin{abstract} A theorem of Hukuhara, Levelt, and Turrittin states that every formal   differential operator has a Jordan decomposition. This theorem was generalised by Babbit and Varadarajan to the case of formal $G$-connections where $G$ is a semisimple group. In this paper, we provide straightforward proofs of these facts,  highlighting the analogy between the linear and differential settings.  
 \end{abstract} 


\maketitle
\tableofcontents

\section{Introduction} \label{s:intro} Let $\cK:=\bC(\!(t)\!)$ be the field of formal   Laurent series and consider the derivation $d:\cK\ra \cK$ defined by $d:=t\frac{d}{dt}$. Let $V$ be a finite dimensional vector space over $\cK$. A \emph{formal   differential operator} is a $\bC$-linear map $D:V\ra V$ satisfying the Leibniz rule 
\begin{equation}\label{eq:diffOp}
D(av)=aD(v)+d(a)v,\quad \quad a\in \cK, \quad v\in V.
\end{equation}

It is well-known that linear operators encode linear equations. Similarly, differential operators encode (ordinary) differential equations. Thus, the study of formal differential operators is indispensable in the theory of meromorphic differential equations; see \cite{Varadarajan} for an extensive review. 

In analogy with linear operators, differential operators have matrix presentations and it will be convenient to have these at our disposal. Indeed, choosing a basis for $V$, we can represent $D$ as an operator $d+A$ where $A$ is an $n\times n$ matrix with values in $\cK$. Changing the basis by an element $g\in GL_n(\cK)$ amounts to changing the operator $d+A$ to $d+g^{-1}Ag + g^{-1}dg$. 
Here $dg$ denotes the matrix obtained by applying the derivation $d$ to each entry of the matrix $g$. The map 
\[
A\mapsto g^{-1}Ag+g^{-1}dg
\]
 is called \emph{gauge transformation} and plays an important role in the theory.

 \subsection{Semisimple Connections} To formulate a Jordan decomposition, we need a notion of semisimplicity. We start with a definition for formal differential operators.
 
 \begin{defe} Let $D:V\ra V$ be a formal differential operator. Then $D$ is
 \begin{enumerate} 
 \item[(i)]   \emph{simple} if $V$ has no $D$-invariant subspace
 \item[(ii)]  \emph{semisimple} if every $D$-invariant subspace has a $D$-invariant complement
 \item[(iii)]  \emph{diagonalisable} if it has a presentation of the form $d+A$ where $A$ is a diagonal matrix
 \item[(iv)]  \emph{potentially diagonalisable} if it is diagonalisable after a finite base change. 
\end{enumerate} 
\end{defe}

 It is easy to show that an operator is semisimple if and only if it is a direct sum of simple ones. The following theorem gives an explicit description of semisimple operators.

\begin{thm}[Levelt] \label{t:semisimple} 
A formal differential operator is semisimple if and only if it is potentially diagonalisable. 
\end{thm}

For future use, we will need the following functorial property. Let $D:V\to V$ be a differential operator and write $D=d+A$ with $A\in \gl(V)$. Consider the adjoint map  
\[
\ad: \gl(V) \ra \gl(\gl(V))
\]
Then $\ad(A)$ is a linear operator on $\gl(V)$; therefore, $d+\ad(A)$ is a differential operator on $\gl(V)$. The following observation will be useful. 

\begin{prop} \label{p:semisimple} The differential operator $d+A$ is semisimple if and only if the differential operator $d+\ad(A): \gl(V)\ra \gl(V)$ is semisimple. 
\end{prop}

 \subsection{Jordan Decomposition} We are now ready to discuss the notion of Jordan decomposition.
 
 \begin{thm}[Hukuhara-Levelt-Turrittin] 
 Every formal   differential operator $D$ can be written as a sum $D=S+N$ of a semisimple differential operator $S$ together with a nilpotent $\cK$-linear operator $N$ such that $S$ and $N$ commute (as $\bC$-linear maps). Moreover, the pair $(S,N)$ is unique. \label{t:main}
\end{thm}

The above theorem has numerous applications in the theory of differential operators and other areas of mathematics, cf. \cite{KatzNilpotent, Katz,   Luu, BoalchYamakawa, KamgarpourSage}.
The existence result was first proved by Turrittin \cite{Turrittin}, building on earlier work of Hukuhara \cite{Hukuhara}. Turrittin's argument was rather complicated involving nine different cases. Subsequently, Levelt gave a more conceptual (albeit still not straightforward) proof and formulated the correct uniqueness  statement \cite{Levelt}. As a corollary, he concluded:

\begin{cor} \label{c:main}
Every formal   differential operator has, after an appropriate finite base change,  an eigenvalue. 
\end{cor} 

 Levelt asked for a direct proof of this corollary, noting that this would considerably simplify the proofs of the above theorems. Subsequently, several authors provided alternative approaches to these theorems cf. \cite{Wasow, Malgrange, Robba, BV, Praagman, SingerVanDerPut, Kedlaya}. One of our main goals is to provide an elementary proof of the fact that every differential operator has an eigenvalue and use it to provide a simple proof of the existence of Jordan decomposition, thus fulfilling Levelt's vision. 
  
 We now provide a brief summary of our approach. Let $\cK\{x\}$ denote the non-commutative ring of \emph{differential polynomials}. As an abelian group $\cK\{x\}=\cK[x]$ but multiplication is modified by the rule $xa=ax+da$ for all $a\in \cK$. Using a version of Hensel's lemma and Newton polygons, we prove: 

\begin{thm} \label{t:factorisation} 
Every non-constant differential polynomial in $\cK\{x\}$ has a linear factorisation over a finite extension of $\cK$. 
\end{thm} 

The above result is established in \S \ref{s:factorisation}. 
Note that Malgrange \cite{Malgrange}  and Robba \cite{Robba}  also use Newton polygons and differential Hensel's lemma in their treatment of the Hukuhara-Levelt-Turrittin Theorem; however, our formulation and proof of Jordan decomposition is different from theirs; for instance, we do not use the cyclic vector lemma.\footnote{For the advantages and disadvantages of the cyclic vector lemma, cf. \cite[\S 5.7]{Kedlaya}.} 

In \S \ref{s:proofs}, we show that  Theorem \ref{t:semisimple} and Corollary \ref{c:main} follow easily from Theorem \ref{t:factorisation}, thus illustrating the analogies between linear and differential setting. Using these results, we obtain a generalised eigenspace decomposition for differential operators. In other words, we obtain that every differential operator has a representation $d+X$ where $X$ is a block-upper triangular matrix and each block has a unique (up to similarity) eigenvalue. At this point, we encounter a subtle difference between the linear and differential setting. Let us write $X=Y+Z$ where $Y$ is diagonal and $Z$ is strictly upper triangular. If we were considering  linear operators, then $Y$ would be the semisimple  and $Z$ the nilpotent part of $X$ and these two commute. In the differential setting, however, the situation is more subtle because the operators $d+Y$ and $Z$ do not necessarily commute.  In fact, these two operators commute if and only if the entries of $Z$ are complex numbers (i.e. have no powers of $t$). We prove that indeed we can arrange so that the entries of $Z$ are complex numbers by using Katz's classification of unipotent differential operators \cite{Katz}.

\subsection{Formal $G$-connections} 
 The above considerations have a natural generalisation to the setting of  algebraic groups. Let $G$ be a connected, semisimple, linear algebraic group over $\bC$ and let $\fg$ denote its Lie algebra.
 A \emph{formal $G$-connection} is an expression of the form 
 \[
 \nabla=d+A, \quad \quad A\in \fg(\cK):=\fg\otimes \cK.
 \]
 
 The group $G(\cK)$ acts on the space of connections by gauge transformation 
 \[
 g\cdot (d+A)=d+{\Ad}_g(A) + (dg)g^{-1}, \quad \quad g\in G(\cK), \quad A\in \fg(\cK). 
 \]
 One way to make sense of the expression $(dg)g^{-1}$ is to choose a faithful representation $\rho: G\ra \GL_n$ (e.g. the adjoint representation) and show that $d(\rho(g)).\rho(g)^{-1}$, a priori in $\gl_n(\cK)$, actually belongs to $\fg(\cK)$, and is independent of the chosen representation; see \cite{BV}*{\S 1.6}, \cite{frenkel}*{\S 1.2.4}, \cite{raskin}*{\S 1.12}.   

\subsubsection{Semisimple $G$-connections} To discuss Jordan decomposition, we first need a notion of semisimplicity for formal $G$-connections. 
Proposition \ref{p:semisimple} allows us to define such a notion:

\begin{defe} \label{d:semisimple} A $G$-connection $\nabla=d+A$, $A\in \fg(\cK)$, is called \emph{semisimple} if $d+\ad(A)$ is semisimple (as a  formal $\GL(\fg)$-connection). 
\end{defe}  

The above is analogous to the definition of ad-semisimplicity for elements in a semisimple Lie algebra, cf. \cite{humphreys}*{\S 5.4}. 
Next, let $H\subseteq G$ be a maximal (complex) torus and $\fh:=\mathrm{Lie}(H)$  the corresponding Cartan subalgebra. We then have an analogue of Theorem \ref{t:semisimple}:   

\begin{thm} \label{t:gsemisimple} A $G$-connection $\nabla=d+A$ is semisimple if and only if, after a finite base change $\cK'/\cK$, $\nabla$ is gauge equivalent to a connection of the form $d+X$ where $X\in \fh(\cK')$. 
\end{thm}

As far as we know this is the first time the above natural theorem has been formulated in the literature. We use properties of the differential Galois group to establish the above theorem; see \S \ref{s:GConnections}.

\subsubsection{Jordan decomposition} We are ready to state Jordan decomposition for formal $G$-connections. 

\begin{thm}[Jordan decomposition] \label{t:main2} 
Every $G$-connection $\nabla=d+A$ can be written as a sum $\nabla=S+N$, where $S$ is a semisimple $G$-connection, $N\in \fg(\cK)$ is a nilpotent element and $S$ and $N$ commute. Moreover, the pair $(S, N)$ is unique. 
\end{thm}

When we say $S$ and $N$ commute, we mean they commute as elements of the extended loop algebra $\hat{\g}=\fg(\cK)\oplus \bC d$, where the bracket is defined by 
\[
[(x\otimes p(t), \alpha.d), (y\otimes q(t), \beta.d)]:= [x\otimes p(t),y \otimes q(t)]+\alpha y\otimes d(q(t))-\beta x\otimes d(p(t)), 
\]
with $x,y\in \fg, \, \, p(t), q(t)\in \cK, \,\, \alpha, \beta\in \bC$.

Following a suggestion of Deligne, Babbit and Varadarajan proved an equivalent form of the above theorem in \cite{BV}. Their proof, which uses intrinsic properties of algebraic groups, is the only proof of this fundamental result available in the literature. In this note, we give an alternative proof which uses the adjoint representation and reduces the problem to the $\GL_n$-case. Our approach is thus similar to the standard proofs of (usual) Jordan decomposition for semisimple Lie algebras, cf. \cite{humphreys}. We refer the reader to \S \ref{s:GConnections} for details.

\subsection{Acknowledgment}  We thank Philip Boalch,  Peter McNamara, Daniel Sage, Ole Warnaar, and Sinead Wilson for helpful conversations. We are grateful to Claude Sabbah for sending copies of \cite{Malgrange} and \cite{Sabbah}. The material in this paper forms a part of SW's Master's Thesis. MK was supported by an ARC DECRA Fellowship.

%%%%%%%%%%%%%%%%%%%%%%%%%%%%%%%%%%%%%%%%%%%%%%%%%%

\section{Factorisation of differential polynomials}\label{s:factorisation} The goal of this section is to prove Theorem \ref{t:factorisation}. This theorem should be thought of as a differential analogue of a classical theorem of Puiseux. As in Section \ref{s:intro}, we consider the differential field $\cK:=\Ct$ with derivation $d$. An important implication of Puiseux's theorem is that for every positive integer $b$, $\cK_b:=\bC(\!(t^{\frac{1}{b}})\!)$ is the unique extension of $\cK$ of degree $b$. The derivation $d$ extends canonically to a derivation $d_b$ on $\cK_b$. 

Let $R$ be a $\mathbb{C}$-algebra and $d:R\rightarrow R$ a derivation. We denote by $R\{x,d\}$ the ring of differential polynomials over $(R,d)$.  We will generally be interested in the cases $R=\cO:=\CT$ and $R=\cK:=\bC(\!(t)\!)$ with derivation of the form $\delta_m:=t^m\frac{d}{dt}$, for some positive integer $m$. According to \cite{Ore}, the ring $\cK\{x,\delta_m\}$ is a left and right principal ideal domain.

\subsection{Differential Hensel's Lemma} \label{s:DiffPoly} 

Let $f\in \cO\{x,\delta_m\}$ be a differential polynomial. We write $f\pmod{t^n}$ for the polynomial obtained by first moving all factors of $t$ to the left and then reducing the coefficients modulo $t^n$. We denote $f\pmod{t}$ by $\bar{f}$. Note that this is a polynomial in $\bC[x]$. Without loss of generality, we assume throughout that $\bar{f}\neq 0$.

Now suppose we have a factorisation of the form
	\[
	\bar{f}=g_0h_0, \quad \quad \, g_0,h_0\in \mathbb{C}[x].
	\]
	Our aim is to lift this to a factorisation of $f$ in $\cK\{x,\delta_m\}$. We think of the following result as a differential analogue of Hensel's lemma. 

\begin{prop}	\label{p:Hensel} Let $f\in \cO\{x,\delta_m\}$ and $\bar{f}=g_0h_0$ as above. Suppose that
\[
\begin{cases} 
\gcd\big(g_0(x+n),h_0(x)\big)=1, \quad \forall n\in \bZ_{> 0} & \textrm{if $m=1$} \\
\gcd\big(g_0(x),h_0(x)\big)=1 & \textrm{if $m>1$}.
\end{cases}  
\]
Then we have a factorisation $f=gh$ with $g,h\in \cO\{x,\delta_m\}$,	 $\operatorname{deg}(g)=\operatorname{deg}(g_0)$,
$\bar{g}=g_0$ and $\bar{h}=h_0$. 
\end{prop}

We note that a version of this proposition appeared in \cite{Praagman}*{Lemma 1}. 

\begin{proof} 
First of all, in the differential polynomial ring $\cK\{x,\delta_m\}$, easy induction arguments show that 
\begin{equation} \label{lemtswap}
	h(x)t^i=t^ih(x+it^{m-1}) ,\quad \quad \forall h(x)\in \cK\{x,\delta_m\},\quad \forall i\in \bZ,
\end{equation}
and
\begin{equation}\label{lempowers} 
	(t^dx)^k=\sum_{j=0}^{k-1}a_jt^{kd+(m-1)j}x^{k-j}, \quad \quad \forall d\in \mathbb{Z}-\{0\}, \quad \forall k\in \mathbb{N},
\end{equation} 
for some constants $a_j\in \mathbb{C}$, $a_0=1$.

Our goal is to inductively build a sequence of polynomials
\begin{align}
\label{eqngn} g_n(x)&=g_0+tp_1+t^2p_2+\dots +t^{n-1}p_{n-1}+t^np_n, \quad \quad \, p_i\in \mathbb{C}[x] \\
\label{eqnhn} h_n(x)&=h_0+t q_1+t^2q_2+\dots +t^{n-1}q_{n-1}+t^nq_n, \quad \quad \, q_i\in \mathbb{C}[x],
\end{align}
which satisfy:
\begin{equation}
f\equiv g_n(x)h_n(x) \pmod{t^{n+1}}. \nonumber
\end{equation}
If we can do this, then by letting $n\to \infty$ we will obtain elements $g,h\in \cO\{x,\delta_m\}$ such that $f=gh$. 

Suppose that we know the $p_i$ and $q_i$ for $1\leq i \leq n-1$. In view of \eqref{eqngn} and \eqref{eqnhn} we have:
\[
g_n=g_{n-1}+t^np_n, \quad \quad h_n=h_{n-1}+t^nq_n.
\]
Requiring that $f\equiv g_n(x)h_n(x) \pmod{t^{n+1}}$ then gives us the following condition:
\begin{align}
f&\equiv g_n(x)h_n(x) \quad &&\pmod{t^{n+1}} \nonumber \\
&\equiv \big(g_{n-1}(x)+t^np_n(x)\big)\big(h_{n-1}(x)+t^nq_n(x)\big) \quad &&\pmod{t^{n+1}} \nonumber \\
&\equiv g_{n-1}(x)h_{n-1}(x)+g_{n-1}(x)t^nq_n(x)+t^np_n(x)h_{n-1}(x)+t^np_n(x)t^nq_n(x) \quad &&\pmod{t^{n+1}}. \nonumber
\end{align}
We need to shift the powers of $t$ to the left. By (\ref{lemtswap}), $g_{n-1}(x)t^n=t^ng_{n-1}\left(x+nt^{m-1}\right)$, so we have:
\begin{align}
f-g_{n-1}(x)h_{n-1}(x)&\equiv t^ng_{n-1}\big(x+nt^{m-1}\big)q_n(x)+t^np_n(x)h_{n-1}(x) \quad &&\pmod{t^{n+1}} \nonumber \\
&\equiv t^n\big(g_{n-1}(x+nt^{m-1})q_n(x)+p_n(x)h_{n-1}(x)\big) \quad &&\pmod{t^{n+1}}, \nonumber
\end{align}
and thus
\begin{align}
\frac{f-g_{n-1}(x)h_{n-1}(x)}{t^n}&\equiv g_{n-1}(x+nt^{m-1})q_n(x)+p_n(x)h_{n-1}(x) \quad &&\pmod{t} \nonumber \\
&\equiv g_0(x+nt^{m-1})q_n(x)+p_n(x)h_0(x) \quad &&\pmod{t}. \nonumber
\end{align}
For notational convenience, we set:
\begin{equation}
f_n=\frac{f-g_{n-1}(x)h_{n-1}(x)}{t^n}. \nonumber
\end{equation}
so that we have
\begin{equation}
\label{eqncondgnhn} f_n\equiv g_0(x+nt^{m-1})q_n(x)+p_n(x)h_0(x) \quad \pmod{t}.
\end{equation}

Now if $m>1$, then (\ref{eqncondgnhn}) reduces to
\[
f_n\equiv g_0(x)q_n(x)+p_n(x)h_0(x) \quad \pmod{t}.
\]
Since $\mathbb{C}[x]$ is a Euclidean domain, we will be able to solve this for $p_n$ and $q_n$ provided that $g_0$ and $h_0$ are coprime. On the other hand, if $m=1$, then (\ref{eqncondgnhn}) becomes
\[
f_n\equiv g_0(x+n)q_n(x)+p_n(x)h_0(x) \quad \pmod{t}.
\]
In this case, we will only be able to generate the entire sequence if $g_0(x+n)$ and $h_0(x)$ are coprime for all $n\in \mathbb{Z}_{>0}$. 

All that remains to show is that we can control the degree of the $g_n$'s. We will show this in the case $m=1$. The proof in the case $m>1$ is similar (replace $g_0(x+n)$ with $g_0(x)$ everywhere).  Since $g_0(x+n)$ and $h_0(x)$ are coprime, we can find $a,b\in \mathbb{C}[x]$ such that
	\begin{equation}
	g_0(x+n)a(x)+h_0(x)b(x)= 1. \nonumber
	\end{equation}
	Multiplying through by $f_n$ yields
	\begin{equation}
	\label{eqnfn2}
	g_0(x+n)a(x)f_n(x)+h_0(x)b(x)f_n(x)=f_n(x). 
	\end{equation}
	Using the division algorithm we can find unique $p_n$ and $q_n$ such that $\operatorname{deg}(p_n)<\operatorname{deg}(g_0)$. Write:
	\begin{equation}
	b(x)f_n(x)= Q(x)g_0(x)+R(x) \nonumber
	\end{equation}
	with $\operatorname{deg}(R)<\operatorname{deg}(g_0)$. Equation (\ref{eqnfn2}) then becomes:
	\begin{equation}
	g_0(x+n)\big(a(x)f_n(x)+Q(x)h_0(x)\big)+h_0(x)R(x)\equiv f_n(x) \quad \pmod{t}. \nonumber
	\end{equation}
	Setting $p_n= R$ and  $q_n= af_n+Qh_0$ gives us the required $g_n$ and $h_n$.
\end{proof}

\begin{cor}
	\label{corregsingfactors}
	Let $f\in \cO\{x,\delta_1 \}$ be a monic differential polynomial. Then $f$ admits a factorisation of the form
	\[
	(x-\Lambda)h,
	\]
	with $\Lambda\in \cO$ and $h\in \cO\{x, \delta_1\}$.
\end{cor}
\begin{proof}
	Let $\bar{f}\in \mathbb{C}[x]$ be the reduction of $f$ mod $t$. Since $f$ is monic, $\bar{f}$ is non-constant and hence factors over $\mathbb{C}$ into linear factors:
	\begin{equation}
	\bar{f}=(x-\lambda_1)(x-\lambda_2)\cdots (x-\lambda_n), \quad \, \lambda_i\in \mathbb{C}. \nonumber
	\end{equation}
	Without loss of generality, we can order these factors so that $\operatorname{Re}(\lambda_1)\le \operatorname{Re}(\lambda_2) \le \cdots \le \operatorname{Re}(\lambda_n)$. With this ordering we then have
	\begin{equation}
	\bar{f}=g_0h_0, \nonumber
	\end{equation}
	where
	\begin{equation}
	g_0=x-\lambda_1, \quad \, h_0=(x-\lambda_2)\cdots (x-\lambda_n). \nonumber
	\end{equation}
	By our choice of ordering, $g_0(x+n)$ has no common factor with $h_0$ for all $n\in \mathbb{Z}_{>0}$. Hence we can apply Proposition \ref{p:Hensel} to obtain a factorisation of the form
	\[
	f=(x-\Lambda)h, \quad \, \Lambda\in \cO, h\in \cO\{x,\delta_1\}, 
	\]
	as required.
\end{proof}

\begin{rem} \label{r:factorisation} 
Note that the above result is false for the usual polynomial ring $\cO[x]$. Indeed, $x^2+t-t^2$ does not have a linear factorisation over this ring, but if we consider it as an element of $\cO\{x,\delta_1\}$, then $x^2+t-t^2=(x-t)(x+t)$. 
\end{rem}

\subsection{From power series to Laurent series via Newton polygons} \label{s:ChangeofVar} In the previous section, we settled linear factorisation for differential polynomials in $\cO\{x,\delta_1\}$. In this section, we explain how, by a change of variable, we can transform polynomials with coefficients in $\cK$ to those with power series coefficients. The price is that we have to go to a finite extension $\cK_q$ of $\cK$ and, more seriously, the derivation is not simply the canonical extension of $\delta_1$ to $\cK_q$. Nevertheless, we shall see that this change of variable allows us to factor elements of $\cK\{x,\delta_1\}$. Throughout we let $v_t(\cdot)$ denote the $t$-adic valuation on $\cK$.

\begin{lem} \label{lemcov2} 
Consider the monic differential polynomial $f(x)=x^n + \sum_{i=1}^n a_i x^{n-i}\in \cK\{x, \delta_1\}$. Let $
r:=\min\big\{\frac{v_t(a_i)}{i}\big\}$. Then $g(X)=t^{-nr}f(t^rX)$ is a monic differential polynomial with power series coefficients. 
\end{lem}
\begin{proof}
	To be more precise, write $r=\frac{p}{q}$ with $\gcd(p,q)=1$ and $q>0$. If $r\ge 0$, then each $v_t(a_i)\ge 0$ and so $f\in \cO\{x,\delta_1\}$. Since we have already dealt with this case in Corollary \ref{corregsingfactors}, we may assume that $r<0$. In order to make the change of variables $x=t^rX$, we require a field extension to $\cK_q=\bC(\!(t^{1/q})\!)$. Let $s:=t^{1/q}$ so that our change of variables becomes $x=s^pX$. Note that this change of variables means that the relation $xt=tx+t$ becomes $Xs^q=s^qX+s^{q-p}$. Hence differential polynomials in $X$ lie in the ring $\cK_q\{X,\frac{1}{q}s^{1-p}\frac{d}{ds}\}$ (note this new derivation sends $s^q$ to $s^{q-p}$). 
	
	Applying (\ref{lempowers}) to $f(s^pX)$ yields $f(s^pX)=s^{np}g(X)$ where 
	\[
	g(X)=s^{-np}a_n + \sum_{k=0}^{n-1}a_k\sum_{j=0}^{n-1-k}m_{n-k,j}s^{(-j-k)p}X^{n-1-k-j}, \quad a_0=1.
	\]
	
	Let $v_s(\cdot)$ denote the $s$-adic valuation on $\cK_q$. Since $v_s(a_i)=qv_t(a_i)$,	$v_t(a_i)\ge \frac{ip}{q}$ implies that $v_s(a_i)\ge ip$.
	Thus, for $0\le l \le n-1$, the coefficient, $b_l$, of $X^{n-l}$ in $g$ satisfies
	\[
	v_s(b_l)=\min_{0\le k \le l}\{v_s(a_ks^{-lp})\}\ge \min_{0\le k \le l}\{kp-lp\}=0,
	\]	
	where the last equality follows since $p<0$. 
	
	It is clear that $v_s(b_l)$ will be $0$ exactly when $v_s(a_l)=lp$, that is, if, and only if, $v_t(a_l)=lr$. For the ``constant'' term of $g$ we have
	\[
	v_s(b_n)=v_s(a_ns^{-np}) \ge np-np=0, \nonumber
\]
	again with equality exactly when $v_t(a_n)=nr$. Thus
	\[
	g(X)=X^n+b_1X^{n-1}+\dots + b_n, \quad \quad b_i\in \CS,
	\]
    with $\min(v_s(b_i))=0$. Furthermore, $v_s(b_i)=0$ if, and only if, $v_t(a_i)=ir$. This shows that $g(X)\in \CS\{X,\frac{1}{q}s^{1-p}\frac{d}{ds}\}$.
\end{proof}

Consider $g(x)$ from the above lemma. If $\bar{g}(x)$ has two distinct roots, then Hensel's lemma allows us to factor it. We now study the opposite extreme, i.e., when all roots of $\bar{g}(x)$ are equal. It will be helpful to use the notion of Newton polygons for differential polynomials, cf. \cite[\S 6.4]{Kedlaya}. Throughout the rest of this section, we will assume that $r<0$, unless explicitly stated otherwise.

\begin{defe}[Newton Polygon]
Let $f\in \cK\{x, \delta_m\}$ be a differential polynomial and write
\[
f(x)=\sum_{i=0}^n a_ix^{n-i}, \quad \quad a_i\in \cK.
\] 
Consider the lower boundary of the convex hull of the points
\[
\{(n-i),v_t(a_i):0\le i \le n\}\subset \mathbb{R}^2.
\]
The Newton polygon of $f$, denoted $NP(f)$, is obtained from this boundary by replacing all line segments of slope less than $1-m$ with a single line segment of slope exactly $1-m$. 
\end{defe}

Lemma \ref{lemcov2} now has the following corollary.

\begin{cor} \label{correproots}
	Let $f$ and $g$ be as in Lemma \ref{lemcov2} and suppose that $\bar{g}:=g \pmod{s}=(X+\lambda)^n$, $\lambda\in \mathbb{C}$. Then $\lambda$ is non-zero and the Newton polygon of $f$ has a single \emph{integral} slope.
\end{cor}
\begin{proof}
	As in Lemma \ref{lemcov2}, write
	\[
	  g(X)=X^n+b_1X^{n-1}+\dots+ b_n, \quad \quad b_i\in \CS.
	\]
	Since $\min\{v_s(b_i)\}=0$, $\lambda\ne 0$. Now since, $\lambda\ne 0$, expanding the bracket $(X+\lambda)^n$ shows that $v_s(b_i)=0$ for all $i$ and hence $v_t(a_i)=ir$. Thus, the Newton polygon of $f$ has a single slope of $-r$ and since $v_t(a_1)=r$, $r$ is an integer.
\end{proof}

For future use, we also record the following lemma. 
\begin{lem} \label{l:NPSlope}
	Let $f$ and $g$ be as in Lemma \ref{lemcov2} and suppose that $\bar{g}=(X+\lambda)^n$, $\lambda\in \mathbb{C}$. Then the slopes of the Newton polygon of $f(x-\lambda t^r)$ are all strictly smaller than the slope of the Newton polygon of $f(x)$. 
\end{lem}
\begin{proof}
By Corollary \ref{correproots}, $r$ is an integer and hence no extension of $\cK$ is necessary. Since $\bar{g}=(X+\lambda)^n$, we can write $g$ as
\[
g=(X+\lambda)^n+e_1(X+\lambda)^{n-1}+\dots+e_n, \quad \quad e_i\in \cO,
\]
with $v_t(e_i)>0$ for all $i$. Now
\begin{align}
f(t^rX)&=t^{nr}\big((X+\lambda)^n+e_1(X+\lambda)^{n-1}+\dots+e_n\big) \nonumber \\
\implies f(x)&=t^{nr}\big((t^{-r}x+\lambda)^n+e_1(t^{-r}x+\lambda)^{n-1}+\dots+e_n\big),\nonumber
\end{align}
and hence
\[
f(x-\lambda t^r)=t^{nr}\big((t^{-r}x)^n+e_1(t^{-r}x)^{n-1}+\dots+e_n\big).
\]
Applying (\ref{lempowers}), we have, for $m_{k,l}\in \mathbb{C}$,
\begin{align}
f(x-\lambda t^r)&=t^{nr}\Big(t^{-nr}\sum_{j=0}^{n-1}m_{n,j}x^{n-j}+e_1t^{-(n-1)r}\sum_{j=0}^{n-2}m_{n-1,j}x^{n-1-j}+\dots+e_n\Big) \nonumber \\
&=\sum_{j=0}^{n-1}m_{n,j}x^{n-j}+e_1t^{r}\sum_{j=0}^{n-2}m_{n-1,j}x^{n-1-j}+\dots+t^{nr}e_n. \nonumber
\end{align}
Since $v(e_i)>0$, the valuation of the coefficient of $x^{n-j}$ in $f(x-\lambda t^r)$ is strictly greater than the corresponding coefficient in $f(x)$. This means that the slopes of the Newton polygon for $f(x-\lambda t^r)$ are strictly less than the slope of the Newton polygon for $f(x)$.
\end{proof}

\begin{exam}
In order to illustrate Corollary \ref{correproots} and Lemma \ref{l:NPSlope}, consider the differential polynomial
\[
f_1(x)=x^2+(4t^{-2}+2t^{-1}+2)x+(4t^{-4}+4t^{-3}+t^{-2}+t^{-1}+1).
\]
In this case, $r=-2$ and the change of variables $x=t^{-2}X$ yields
\[
g_1(X)=X^2+(4+2t)X+(4+4t+t^2+t^3+t^4),
\]
and so $\bar{g}_1(X)=(X+2)^2$. The figure below shows that the Newton polygon of $f_1$ has only a single slope of $-2$ (cf. Cor \ref{correproots}). Making the translation $x\mapsto x+2t^{-2}$ as in Lemma \ref{l:NPSlope} yields the new polynomial 
\[
f_2(x)=x^2+(2t^{-1}+2)x+(t^{-2}+t^{-1}+1).
\]
This has a single slope $r=-1$ and a final translation $x\mapsto x-t^{-1}$ yields $f_3(x)=x^2+2x+1$. This can easily be factorised and reversing the change of variables yields the full factorisation $f_1(x)=(x+2t^{-2}+t^{-1}+1)^2$.
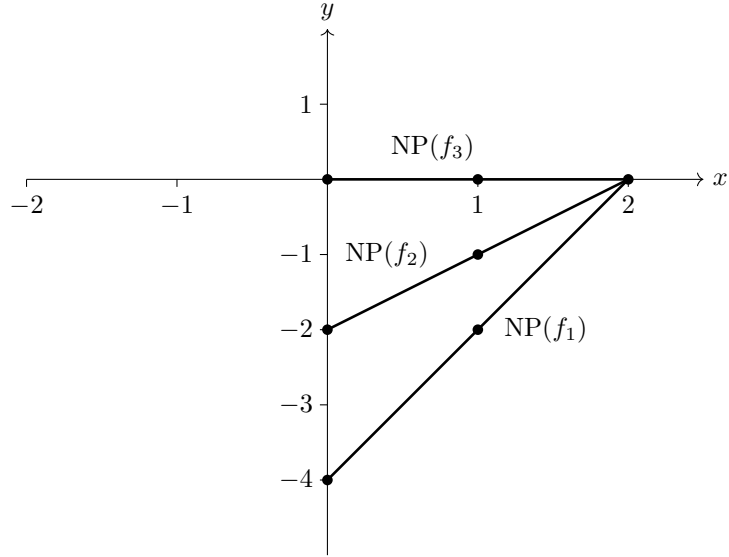
\[
\begin{tikzpicture}[dot/.style={circle,fill=black,minimum size=4pt,inner sep=0pt,
	outer sep=-1pt}]

%draw axes
\draw[->] (0,-5) -- (0,2) node [above] {$y$};
\draw[->] (-4,0) -- (5,0) node [right] {$x$};

%label x axis
\foreach \pos/\label in {-4/$-2$,-2/$-1$,
	2/$1$,4/$2$}
\draw (\pos,0) -- (\pos,-0.1) (\pos cm,-3ex) node
[anchor=base,fill=white,inner sep=1pt]  {\label};

%label y axis
\foreach \pos/\label in {-4/$-4$,-3/$-3$,-2/$-2$,-1/$-1$,
	1/$1$}
\draw (0,\pos) -- (-0.1,\pos) (-1ex,\pos cm) node
[left,fill=white,inner sep=1pt]  {\label};

%create points for NP(f_1)
\node[dot] at (4,0) {};
\node[dot] at (2,-2) {};
\node[dot] at (0,-4) {};
\draw[line width=1pt] (4,0) -- (0,-4) (2.3cm, -2cm) node[right,fill=white,inner sep=1pt]{NP($f_1$)};

%create points for NP(f_2)
\node[dot] at (2,-1) {};
\node[dot] at (0,-2) {};
\draw[line width=1pt] (4,0) -- (0,-2) (1.4cm, -1cm) node[left,fill=white,inner sep=1pt]{NP($f_2$)};

%create points for NP(f_3)
\node[dot] at (2,0) {};
\node[dot] at (0,0) {};
\draw[line width=1pt] (4,0) -- (0,0) (1.4cm, 0.2cm) node[above,fill=white,inner sep=1pt]{NP($f_3$)};
\end{tikzpicture}
\]
\end{exam}

\subsection{Proof of Theorem \ref{t:factorisation}} \label{s:Theorem1Proof}
Write $f(x)=x^n+a_{1}x^{n-1}+\cdots+a_n  \in \cK\{x,\delta_1\}$ and let $r:=\min\big\{\frac{v_t(a_i)}{i}\big\}\in \bQ$. If $r\geq 0$, then the result follows from the differential Hensel's Lemma (see Corollary \ref{corregsingfactors}) so we may assume $r<0$. Let us write 
\[
r=\frac{p}{q}, \quad \quad q>0, \quad \gcd(p,q)=1. 
\]
Consider the transformation $x\mapsto t^{r} X$. Under this transformation the differential field $(\cK, \delta_1)$ changes to $\big(\cK_q, \frac{1}{q}s^{1-p}\frac{d}{ds}\big)$ where $s:=t^{1/q}$. Moreover, we obtain a monic differential polynomial  $g(X)\in \CS\{y,\frac{1}{q}s^{1-p}\frac{d}{ds}\}$. Let $\bar{g}(X)$ denote the reduction of $g(X)$ modulo the maximal ideal of $\CS$. If $\bar{g}(X)$ has two distinct roots, then we can again apply Proposition \ref{p:Hensel}  to reduce the problem to a polynomial of degree strictly less than $f$. Thus, we are reduced to the case that $\bar{g}(X)$ has a unique repeated root $\lambda$. For inductive purposes, we rename $f$ to $f_1$. In this case, by Corollary \ref{correproots}, $\lambda\neq 0$ and the Newton polygon of $f_1$ has a single \emph{integral} slope. Now we make the transformation $x\mapsto x-\lambda t^r$. As shown in Lemma \ref{l:NPSlope}, under this transformation $f_1$ is mapped to a polynomial $f_2$ whose Newton polygon has slopes strictly less than that of $f_1$. Note that this transformation does not change the differential field.

Now we start the process with the polynomial $f_2(x):=x^n+b_1x^{n-1}+\cdots +b_n\in \cK\{x,\delta_1\}$; i.e., we let $r_2:=\min\big\{\frac{v(b_i)}{i}\big\}$. If $r_2\geq 0$ we are done. Otherwise, we make the change of variable $x\mapsto t^{r_2} X$ to obtain a new polynomial $g_2(X)$. If $\bar{g}_2(X)$ has distinct roots, then we are done; otherwise, applying Corollary \ref{correproots} again, we conclude that the Newton polygon of $f_2$ has a single \emph{integral} slope. Since the slope of $f_2$ is a nonnegative integer strictly less than slope of $f_1$, this process must stop in finitely many steps at which point we have a factorisation of our polynomial.  

\qed

%%%%%%%%%%%%%%%%%%%%%%%%%%%%%%%%%%%%%%%%%%%%%%%%%%%%%%%%%%%%

\section{Formal differential operators} \label{s:proofs} Recall that for each positive integer $b$,  $\cK_b$ denotes the unique finite extension of $\cK$ of degree $b$. 
Given a differential operator $D$, one has a canonical differential operator 
\[
D\otimes_\cK \cK_b: V\otimes_\cK \cK_b \ra  V\otimes_\cK \cK_b
\]
called the base change of $D$ to $\cK_b$. All base changes considered in this article are of this form. 
Henceforth, we will use the notation $V_b:=V\otimes_\cK \cK_b$ and $D_b=D\otimes_\cK \cK_b$. 

\subsection{Proof of Corollary \ref{c:main} (Every differential operator has an eigenvalue)}
 The argument proceeds exactly as in the linear setting. Let $D:V\ra V$ be a differential operator and $v\in V$ be a non-zero vector. Consider the sequence $v, D(v), D^2(v),\cdots$. As $V$ has finite dimension over $\cK$, we must have that 
\[
D^n(v) + a_1D^{n-1}(v)+\dots + a_{n-1}D(v)+a_nv=0,\quad \quad a_i\in \cK,
\]
where $n=\dim_\cK(V)$. Now consider the polynomial $f(x)=x^n+a_1x^{n-1}+\dots + a_n$ in the twisted polynomial ring $\cK\{x\}$. After a finite extension, we can write 
\[
f(x)=(x-\Lambda_1)\cdots (x-\Lambda_n)\in \cK_b\{x\}, \quad \Lambda_i\in \cK_b, b\in \bZ_{>0}.
\]
Thus,
\[
(D_b-\Lambda_1)\cdots (D_b-\Lambda_n)v=0.
\]
Let $i\in \{1,2,\cdots, n\}$ be the largest number such that $(D_b-\Lambda_i)\cdots (D_b-\Lambda_n)v=0$. If $i=n$ , then $v$ is an eigenvector of $D_b$ with eigenvalue $\Lambda_n$. Otherwise $(D_b-\Lambda_{i+1})\cdots (D_b-\Lambda_n)v$ is an eigenvector of $D_b$ with eigenvalue $\Lambda_i$.
\qed

\subsection{Proof of Theorem \ref{t:semisimple} (Semisimple operators are diagonalisable)} \label{s:proofoftsemisimple}
We need the following lemma. The proof is an easy argument using the Galois  group $\Gal(\cK_b/\cK)$; see \cite[\S 1(e)]{Levelt} for details. 
\begin{lem} \label{l:semisimple} $D$ is semisimple if and only if $D_b$ is. 
\end{lem} 

Now we are ready to prove Theorem \ref{t:semisimple}. 
Suppose $D$ is semisimple. 
We prove by induction on $\dim(V)$ that, after an appropriate base change, it is diagonalisable. If $\dim(V)=1$ the result is obvious, so assume $\dim(V)>1$. Without loss of generality, assume $D$ has an eigenvector $v$ (if not, do an appropriate base change; by the previous lemma, the operator remains semisimple). Let $U=\Span_{\cK}\{v\}$. Then $U$ is a one-dimensional, $D$-invariant subspace\footnote{Indeed, if $Dv=\lambda v$, $\lambda \in \cK$, and $a\in \cK$ then $D(av)=(a\lambda+d(a))v\in U$.} of $V$; thus, there exists a $D$-invariant complement $W$. Now $D:W\to W$ is semisimple so by our induction hypothesis (after an appropriate base change), we can write $W$ as a direct sum of one-dimensional subspaces. Thus, after an appropriate base change, we have a decomposition of our vector space into one-dimensional, invariant subspaces and so $D$ is diagonalisable.  

Conversely, suppose $D_b$ is a diagonalisable operator. Then clearly $D_b$ is semisimple and thus, by Lemma \ref{l:semisimple}, so is $D$. \qed

\subsection{Invariant Properties of Differential Operators} \label{s:diffops}
The goal of this section is to prove Proposition \ref{p:semisimple}. To this end, we need to establish some properties of differential operators.

\subsubsection{Invariant Subspaces}

\begin{lem} \label{l:invariance} Let $D:V\to V$ be a differential operator with Jordan decomposition $D=S+N$. Suppose that $W\subset V$ is a $D$-invariant subspace. Then $W$ is also $S$-invariant. 
\end{lem}

\begin{proof}
Note that $V$ decomposes into generalised eigenspaces and that these generalised eigenspaces are $D$, $S$ and $N$ invariant \cite{Levelt}*{\S 4}. Hence we need only consider the case where $V$ itself is a generalised eigenspace. In this case, there exists a finite extension, $\cK_b$, of $\cK$ such that $S=d+\lambda I$ for some $\lambda\in \cK_b$. 
We first prove the result in the case of unipotent differential operators (i.e. in the case $\lambda=0$). As in Section \ref{s:unipotent}, we denote by $\cU$ the category of unipotent differential operators. Recall this category is equivalent to the category $\Nilp$ whose objects are pairs $(V_0,N)$ where $V_0$ is a $\bC$-vector space and $N$ is a nilpotent linear operator.

The restriction $D|_W:W\to W$ gives us a monomorphism in the category $\cU$. Under the equivalence $F$ we obtain a monomorphism in $\Nilp$.
Hence, there is a basis of $V$ for which we can write $D=d+N$. Since $(W_0,N')\hookrightarrow (V_0,N)$, in this basis we have $dW=W$. That is, $W$ is $S$-invariant. 

This result clearly extends to differential operators with a unique (up to similarity) eigenvalue.

For the general case, recall that after a finite extension to $\cK_b$, we can write $D=S+N$ where $S$ is diagonalisable. Now $W_b$ is a $D_b$-invariant subspace of $V_b$ and so by the above, $W_b$ is also $S_b$-invariant. If $W$ were not $S$-invariant, then $W_b$ would not be $S_b$-invariant, hence $W$ must be $S$-invariant. 
\end{proof}

\subsubsection{Adjoint differential operator} Let $V$ be a finite dimensional vector space over $\cK$.
Given a differential operator $d+A:V\ra V$, we write $d+S+N$ for its Jordan decomposition. Note that $S$ is not necessarily a semisimple linear operator on $V$; rather, $d+S$ is a semisimple differential operator on $V$.

\begin{lem} \label{l:jordan} Let $d+A:V\ra V$ be a differential operator, where $A\in \gl(V)$. Then $d+\ad S + \ad N$ is the Jordan decomposition of $d+\ad A$. 
\end{lem} 

\begin{proof}There exists a finite extension $\cK_b$ of $\cK$ such that we can pick a basis for $V\otimes \cK_b$ to put $d+A$ in Jordan normal form. In this case, $S$ is diagonal and $N$ is a constant nilpotent matrix with $1$'s or $0$'s on the super-diagonal, and $S$ and $N$ commute. Thus, $d+\ad(S)$ is a semisimple differential operator on $\gl(V\otimes \cK_b)$. We claim that it commutes with $\ad(N)$. Indeed, 
\[
[d+\ad(S), \ad(N)]=[d, \ad(N)]+[\ad(S), \ad(N)],
\]
where the bracket is for the extended Lie algebra $\widehat{\gl(\g)}$. Now $\ad(N)$ is constant, so the first bracket is zero. Since $S$ and $N$ commute, the second bracket is also zero. 
\end{proof} 

\subsubsection{Proof of Proposition \ref{p:semisimple}}
 If $d+A$ is semisimple, then we have seen that so is $d+\ad(A)$. If $d+A$ is not semisimple, then suppose $d+S+N$ is its Jordan decomposition. By assumption, $N\neq 0$. This implies that $\ad N$ is not trivial. Thus, $d+\ad A$ is not semisimple.

\subsection{Generalised eigenspace decomposition} \label{ss:eigenspace} 
Let $D:V\ra V$ be a formal differential operator and let $a\in \cK$. 

\begin{defe} \label{d:geneigenspace}  The generalised eigenspace $V(a)$ of $D$ is defined as
	\[
	V(a): =\Span_{\cK} \{ v\in V\, | \, (D-a)^nv=0, \quad \textrm{for some positive integer $n$.}\}
	\]
\end{defe} 

The goal of this section is to prove the following theorem. 

\begin{thm}[Generalised eigenspace decomposition] For some finite extension $\cK_b$ of $\cK$ there exists a canonical decomposition $V_b=\bigoplus_{i} V_b(a_i)$. Moreover, 
\[
V_b(a_i)\cap V_b(a_j)\neq \{0\} \iff \textrm{$a_i$ is  similar to $a_j$}\iff V_b(a_i)=V_b(a_j).
\]
\label{t:eigenspace}
\end{thm} 
 
Before proving this theorem, we need to recall some facts about differential operators. Let $D:V\ra V$ be a differential operator. Define 
\[
H^0(V):=\ker(D), 
\]
\[
 H^1(V):=V/D(V).
\]
Note that these are vector spaces over $\bC$ (not over $\cK$). The following proposition due to Malgrange \cite[Theorem 3.3]{Malgrange2} is an analogue of the rank-nullity theorem for formal differential operators. 

\begin{prop}\label{prop:ranknull}
	Let $D:V\to V$ be a formal differential operator. Then 
	\[
	\dim_{\bC}H^0(V)=\dim_{\bC}H^1(V).
	\]
\end{prop} 

Next, recall that 
the dual differential operator $D:V\ra V$ is the operator $D^*$  on the vector space $V^*=\Hom_{\cK}(V,\cK)$ defined by 
\[
D^*:V^*\to V^*, \quad D^*(f)=d\circ f - f \circ D,\quad \quad f\in V^*.
\]

Let $D:V\ra V$ and $D':V'\ra V'$ be differential operators. Then, we can define a differential operator $D\otimes D'$ on $V\otimes V'$ by 
	\[
	(D\otimes D')(v\otimes v'):=D(v)\otimes v'+v\otimes D'(v').
	\]

The set of all of $\cK\{x\}$-linear maps from $V$ to $V'$ is denoted $\Hom_{\cKx}(V,V')$. This is a $\bC$-vector space. 
The Yoneda extension group $\Ext_{\cKx}^1(V,V')$ consists of equivalence classes of extensions of $\cK\{x\}$-modules 
\[
0 \to V \to V'' \to V' \to 0
\]
As usual, two extensions are equivalent if there exists a $\cKx$-linear isomorphism between them inducing the identity on $V$ and $V'$.

\begin{prop} Let $D:V\ra V$ and $D':V'\ra V'$ be two formal differential operators. Then, we have 
\begin{enumerate} 
\item[(i)]
$ \dim_\bC \Ext_{\cK\{x\}}^1(V,V')=\dim_\bC \HH^0(V^*\otimes V')$.
\item[(ii)] If no eigenvalue of $D$ is similar to an eigenvalue of $D'$, then $\Ext_{\cKx}^1(V,V')=0$.  
\end{enumerate} 
\end{prop} 

\begin{proof} One can show (see \cite{Kedlaya}*{Lemma 5.3.3}) that there is a canonical isomorphism of $\bC$-vector spaces:
\[
\Ext_{\cKx}^1(V,V')\simeq H^1(V^*\otimes V').
\]
This fact together with Proposition \ref{prop:ranknull}  implies (i).

The eigenvalues of $D^*\otimes D'$ are of the form $-a+a'$ where $a$ and $a'$ are eigenvalues of $D$ and $D'$, respectively. By assumption, $-a+a'$ is never similar to zero; thus, kernel of $D^*\otimes D'$ is trivial. Part (ii) now follows from Part (i). 
\end{proof}

\begin{proof}[Proof of Theorem \ref{t:eigenspace}]
We may assume, without the loss of generality, that all eigenvalues of $D$ are already in $\cK$ (if not, do an appropriate base change).  We use induction on $\dim(V)$ to prove the theorem. If $\dim(V)=1$ then the claim is trivial. Suppose $\dim(V)>1$. Then by assumption $D$ has an eigenvector.  Hence, we have a one-dimensional invariant subspace $U\subset V$. Let $W:=V/U$. Then $D$ defines a differential operator on $W$. Moreover, $V\in \Ext_{\cKx}^1(U,W)$. By induction we may assume that $W$ decomposes as
\[
W= \bigoplus_i W(a_i), \quad a_i\in \cK,
\]
 for non-similar $a_i$. Now
\[
V\in \Ext_{\cKx}^1\Big(U, \bigoplus_i W(a_i) \Big)\simeq \bigoplus_i \Ext_{\cKx}^1(U, W(a_i)).
\]

If the eigenvalue $a$ of $D|_U$ is not similar to any $a_i$ then by the above proposition all the  extension groups are zero, and so $V=W\oplus U$ and the theorem is established. If $a$ is similar to $a_j$, for some $j$, then the only non-trivial component in the above direct sum is  $\Ext_{\cKx}^1(U, W(a_j))$. But it is easy to see that all differential operators in $\Ext_{\cKx}^1(U, W(a_j))$ have only a single eigenvalue $a_j$ (up to similarity). Hence $V$ has the required decomposition. 
\end{proof}

\subsection{Unipotent differential operators} \label{s:unipotent} Theorem \ref{t:eigenspace} implies that we only need to prove Jordan decomposition for differential operators with a unique eigenvalue. By translating if necessary, we can assume this eigenvalue is zero. Thus, we arrive at the following: 

\begin{defe}[Unipotent Operators] A differential operator is \emph{unipotent} if all of its eigenvalues are similar to zero. 
\end{defe} 

We now give a complete description of unipotent differential operators. 
Let $\Nilp_\bC$ denote the category whose objects are pairs $(V,N)$ where $V$ is a $\bC$-vector space and $N$ is a nilpotent endomorphism. The morphisms of $\Nilp_\bC$ are linear maps which commute with $N$.  Let $\mathcal{U}$ be the category of pairs $(V,D)$ consisting of a vector space $V/\cK$ and a unipotent differential operator $D:V\ra V$. Define a functor 
	\[
	F:\Nilp_\bC\ra \mathcal{U},\quad \quad (V,N)\mapsto (\cK \otimes_\bC V, d+N).
	\]

The following result appears (without proof) in \cite[\S 2]{Katz}. 

\begin{lem} \label{l:unipot} 		 The functor $F$ defines an equivalence of categories with inverse given by
	\[
	G:\mathcal{U}\ra \Nilp_\bC, \quad \quad (V,D)\mapsto \big(\ker(D^{\dim_{\cK}(V)}),D\big).
	\]
	\end{lem} 
	\begin{proof} 
		We first show that the composition $G\circ F$ equals the identity. Let $(V,N)\in \Nilp_\bC$ with $n:=\dim(V)$ and consider $F(V,N)=(V\otimes \cK, d+N)$. The kernel of the  operator $(d+N)^n$ acting on $V\otimes \cK$ is the set of all constant vectors. This is an $n$-dimensional $\bC$-vector space. 
	Since $d$ acts as 0 on this space, applying $G$ to $(\cK\otimes V,d+N)$ recovers the pair $(V,N)$. 
	
	Next, let $D:V\ra V$ be a unipotent differential operator and let $n:=\dim_{\cK}(V)$. We first show by induction that $\ker(D^n)$ contains $n$, $\cK$-linearly independent vectors. If $n=1$ this is obvious. If $n>1$, then there exists $v\in V$ such that $Dv=0$. Set $U:=\Span_{\cK}\{v\}$ and consider the differential module $V/U$. This has dimension $n-1$ so we may assume there exist $\{v_1,\dots, v_n\}$ $\cK$-linearly independent vectors in $\ker(D^{n-1})$. For each $v_i$ we have $D^{n-1}v_i+U=U$ and hence $D^{n-1}v_i=a_iv$ for some $a_i\in \cK$. Now observe that we can choose $b_i$ such that $d^{n-1}(b_i)=a_i-a_{i,0}$ where $a_{i,0}$ is the constant term of $a_i$; since we can always ``integrate'' elements with no constant term. Now we have
	\[
	D^{n-1}(v_i-b_iv) =D^{n-1}v_i-D^{n-1}(b_iv) =
	a_iv-\sum_{j=0}^{n-1}\binom{n-1}{j}d^j(b_i)D^{n-1-j}(v)=
	a_iv -d^{n-1}(b_i)v = a_{i,0}v.
	\]
	Hence $D^n(v_i-b_iv)=D(a_{i,0}v)=0$ so $\{v,v_1-b_1v,\dots,v_{n-1}-b_{n-1}v\}$ is a set of $\cK$-linearly independent vectors in $\ker(D^n)$. 
	
	Note the functor $G$ sends $V$ to the $\bC$-vector space $W:=\ker(D^n)=\Span_{\bC}\{v,v_1-b_1v,\dots,v_{n-1}-b_{n-1}v\}$. Moreover, $D$ induces a $\bC$-linear operator $N$ on $W$. By construction, this operator is nilpotent and for this basis, the matrix of $N$ is constant (i.e., its entries belong to $\bC$). Applying the functor $F$ to $(W, N)$ now recovers the differential module $(V,D)$. 
\end{proof}

\begin{rem} A formal differential operator $D$ is said to be \emph{regular singular} if it has a matrix representation of the form 
\[
A_0+A_1t+\cdots, \quad \quad A_i\in \gl_n(\bC).
\]
It is known that, in this case, $D$ can actually be represented by a constant matrix; i.e., by a matrix $A\in \gl_n(\bC)$. The conjugacy class of $A$ is uniquely determined by $D$ and is called the \emph{monodromy} \cite[\S 3]{BV}.  The above lemma implies that a unipotent differential operator is the same as a regular singular differential operator with unipotent monodromy. 
\end{rem}

\subsection{Proof of Theorem \ref{t:main} (Jordan Decomposition)}
The uniqueness part of the theorem is relatively easy. Since we don't have anything new to add to Levelt's original proof, we refer the reader to \cite{Levelt} for the details. It remains to prove existence. 

Let $D:V\ra V$ be a formal   differential operator. By Theorem \ref{t:eigenspace}, there exists a positive integer $b$ such that $D_b:V_b\ra V_b$ admits a generalised eigenspace decomposition. Thus, $D_b$ can be represented by a block diagonal matrix where each block is upper triangular with a unique (up to similarity) eigenvalue. Thus, we may assume without the loss of generality that $D_b$ has a unique, up to similarity, eigenvalue $a$. Replacing $D_b$ by $D_b-a$, we may assume that $D_b$ is unipotent in which case the result follows from Lemma \ref{l:unipot}. This proves the existence of Jordan decomposition for $D_b$.

We now show that the Jordan decomposition of $D_b$ descends to a decomposition of $D$. The proof is similar to the linear setting. Picking a $\cK$-basis of $V$ and extending it to a basis of $V_b$ allows us to write $D_b=d+A$ where $A$ is a matrix with entries in $\cK$. Let $S_b=d+B$ and $N_b=C$ for matrices $B$ and $C$ with respect to this basis. Then, for any $\sigma \in \Gal(\cK_b/\cK)$, it is clear that $d+A=d+\sigma(B)+\sigma(C)$ is a second Jordan decomposition of $D_b$. Thus, we must have $C=\sigma(C)$ and $\sigma(B)=B$. Hence, $d+B$ and $C$ are defined over $\cK$. 
\qed

%%%%%%%%%%%%%%%%%%%%%%%%%%%%%%%%%%%%%%%%%%%%%%%%%%%
%%%%%%%%%%%%%%%%%%%%%%%%%%%%%%%%%%%%%%%%%%%%%%%%%%%

\section{Formal $G$-connections}\label{s:GConnections}

\subsection{Description of semisimple $G$-connections}  
We start by recalling basic facts about the differential Galois group. 
Let $I_\cK$ denote the differential Galois group of $\cK$ as defined in \cite[\S 2.5]{Katz}. By definition, for every $G$-connection $\nabla$, we get a homomorphism $\rho_\nabla: I_\cK\ra G$.  The Zariski closure of the image of $\rho_\nabla$ is an algebraic subgroup of $G$ called the \emph{differential Galois group} of $\nabla$ and denoted by $G_\nabla$. For an alternative point of view on $G_\nabla$, cf. \cite{SingerVanDerPut}*{\S 1.4}. 

\begin{proof}[Proof of Theorem \ref{t:gsemisimple}]
We are now ready to prove the theorem. 
  Suppose the differential operator $d+A$, $A\in \fg(\cK)$, is gauge equivalent to $d+X$ with $X\in \fh(\cK')$ for some finite extension $\cK'$ of $\cK$. Then $\ad(X)\in \ad(\fh)(\cK')$ and $\ad(\fh)$ is contained in some Cartan subalgebra of $\gl(\fg)$. Then there exists $g\in \GL(\fg)(\bC)$ such that $g^{-1}\ad(\fh)g\in \mathfrak{d}$ where $\mathfrak{d}$ consists of diagonal matrices in $\gl(\fg)$. As $dg=0$,  the gauge action of $g$ on $d+\ad(X)$ yields $d+g^{-1}\ad(X)g$. Thus $d+\ad(X)$ is gauge equivalent to a diagonal differential operator and is therefore semisimple by Theorem \ref{t:semisimple}. As $d+\ad(X)$ is gauge equivalent to $d+\ad(A)$, this implies that $d+\ad(A)$ is semisimple. By definition, then $d+A$ is semisimple. 
	
	Conversely, suppose that $d+A$ is semisimple, i.e. $d+\ad(A)$ is semisimple. By Theorem \ref{t:semisimple}, $d+\ad(A)$ is diagonalizable after a finite extension $\cK'$ of $\cK$. This implies that the image of the composition 
	\[
	I_{\cK'} \ra G\ra \GL(\fg) 
	\]
	is a subgroup of a torus in $\GL(\fg)$. This then implies that the image of $I_{\cK'}\ra G$ is a subgroup of a maximal torus $H\subset G$; that is, the above map factors through a map $I_{\cK'}\ra H$. Thus, $d+A$ is equivalent to a connection of the form $d+X$ for some $X\in \fh(\cK')$. 
\end{proof}

	\begin{rem} 
 Let $\nabla=d+X$ be a semisimple formal $G$-connection. By Theorem \ref{t:gsemisimple}, we may assume (after a finite base change) that $X\in \fh(\cK)$ where $\fh$ is a Cartan subalgebra of $\fg$. Write $X=\sum_i X_it^i$ where $X_i\in \fh(\bC)$ and set 
\[
X^+=\sum_{i\geq 1}X_it^i\in \fh[\![t]\!], \quad \quad X^-=\sum_{i\leq 0} X_it^i\in \fh[t^{-1}].
\]
 Let
\[
A:=\exp\Big( -\int \frac{X^+}{t} dt\Big)\in H(\CT).
\]
Then gauge transformation of $\nabla$ by $A$ yields $d+X^-$. This is the canonical form of  $\nabla$ in the sense of \cite{BV}. 
\end{rem}

\subsection{Jordan decomposition for $G$-connections}
We start with a lemma, which is an analogue of a standard result in Lie theory, c.f. \cite{humphreys}*{\S 6.4}.

\begin{lem} \label{l:preservation} Let $\g\subset \gl(V)$ be a Lie subalgebra. Let $d+A:V\ra V$,  $A\in \fg(\cK)$, be a differential operator with Jordan decomposition (as a $\GL(V)$-connection) $D=d+X+N$. Then $X\in \fg(\cK)$; moreover, $d+X$ is a semisimple $G$-connection.  
\end{lem}
\begin{proof} 
By definition $\fg(\cK)$ is a $(d+\ad A)$-invariant subspace of $\gl(V)\otimes \cK$. Thus, by Lemma \ref{l:invariance}, it is also $(d+\ad X)$-invariant. By definition, $\fg(\cK)$ is $d$ invariant. Thus, $\fg(\cK)$ is $\ad X$-invariant. This implies that $(d+\ad X)-d:\fg(\cK)\to \fg(\cK)$ is a $\cK$-linear derivation on $\g(\cK)$ and hence $X\in \g(\cK)$ (since every $\cK$-linear derivation is inner). 
\end{proof}

\begin{proof}[Proof of Theorem \ref{t:main2}]
We are now ready to prove the theorem. Let $\nabla=d+A$ be a $G$-connection. Note that the adjoint action gives an embedding $\fg(\cK)\subset \gl(\fg(\cK))$. Let $d+X+N$ denote the Jordan decomposition of $d+A$ as a differential operator $\fg(\cK)\ra \fg(\cK)$.  Then by the previous lemma, $X\in \fg(\cK)$ and $d+X$ is a semisimple $G$-connection.  It follows that $N=d+A-(d+X)$ is a nilpotent element of $\fg(\cK)$. Now $d+X$ and $N$ commute in the extended loop algebra of $\gl(\fg(\cK))$. This implies that they commute in the extended loop algebra of $\fg(\cK)$, this establishes the existence of Jordan decomposition. 

For uniqueness, suppose $d+X_1+N_1$ and $d+X_2+N_2$ are Jordan decompositions for $\nabla$. Then $d+\ad(X_1)+\ad(N_1)$ and $d+\ad(X_2)+\ad(N_2)$ are Jordan decompositions for $\ad(\nabla)$. By uniqueness of Jordan decomposition for differential operators (Theorem \ref{t:main}), we obtain $\ad(X_1)=\ad(X_2)$ and $\ad(N_1)=\ad(N_2)$. As the adjoint representation is faithful, we conclude $X_1=X_2$ and $N_1=N_2$. 
\end{proof} 

\begin{rem} A formal $G$-connection $\nabla$ is called \emph{unipotent} if its semisimple part is trivial. One can show that $\nabla$ is unipotent if and only if its differential Galois group $G_\nabla$ is unipotent. According to \cite[thm. 11.2]{SingerVanDerPut}, $G_\nabla$ is then a one-parameter subgroup of $G$ generated by a unipotent element. This implies that the map 
\[
Y \mapsto d+Y
\]
from nilpotent elements of $\fg(\bC)$ to formal unipotent $G$-connections defines a bijection between nilpotent orbits in $\fg(\bC)$ and equivalence classes of unipotent connections. Thus, one obtains a generalisation of Katz's results (Lemma \ref{l:unipot}) to the setting of $G$-connections.
\end{rem}

%\subsection{Canonical form} 
%Let $\cK_b:=\Ctb$, $b\in \mathbb{N}$, be a finite extension of $\cK$. As a consequence of the above theorems, we recover the following fundamental result of Babbit and Varadarajan \cite{BV}*{\S 9}. 
%\begin{thm} After a possible extension to $\cK_b$, every $G$-connection is gauge equivalent to a unique canonical form 
%\[
%d+ X_rz^{-r}+X_{r-1}z^{-r+1} + \cdots + X_1z^{-1}+C, \quad z=t^{1/b} 
%\]
%where $X_i \in \fh$ and $C\in \g$ commutes with each $X_i$. 
%\end{thm} 

%Note that the proof in \cite{BV} is rather complicated. By comparison, ours is quite straight-forward and follows the approach taken for proving Jordan decomposition in semisimple Lie algebras.

%\begin{proof}
%	Let $\nabla=d+A$, $A\in \g(\cK)$, be a $G$-connection. By Theorem \ref{t:main2}, we can write $\nabla$ uniquely as $d+X+N$ with $d+X$ semisimple and $N\in \g(\cK)$ nilpotent. After a finite extension, Theorem \ref{t:semisimple} allows us to put this in Jordan form (i.e. we can gauge transform so that $X\in \frak{h}(\cK_b)$).
%\end{proof}

  \end{document}